\documentclass[11pt]{article}
\usepackage{CJK,amsmath,amsthm,dsfont,amsfonts,amssymb,fancyhdr}
\usepackage{enumerate}
\usepackage{bbm,color,soul,supertabular,longtable,verbatim,extarrows}
\usepackage{titlesec}
\usepackage{graphicx}
\usepackage{cite}
\usepackage{tikz}
\usepackage{booktabs,multirow,makecell}
\usepackage{authblk}
\usepackage{lipsum}
\usepackage{mathrsfs}
\usepackage{indentfirst}
\usepackage[top=2.4cm,bottom=2.2cm,left=2.6cm,right=2cm]{geometry} 

\newcommand{\rk}{\textrm{rk}}
\newtheorem{theorem}{Theorem}[section]
\newtheorem{proposition}[theorem]{Proposition}

\newtheorem{lemma}[theorem]{Lemma}
\newtheorem{example}[theorem]{Example}

\begin{document}
\title{\textbf{Some Remarks on Systems of Equiangular Lines}}
\author[a]{Mengyue Cao}
\author[b,c,]{Jack H. Koolen\footnote{Corresponding author.}}
\author[d]{Jae Young Yang}
\affil[a]{\footnotesize{School of Mathematical Sciences, Beijing Normal University, 19 Xinjiekouwai Street, Beijing, 100875, PR China.}}
\affil[b]{\footnotesize{School of Mathematical Sciences, University of Science and Technology of China, 96 Jinzhai Road, Hefei, 230026, Anhui, PR China.}}
\affil[c]{\footnotesize{Wen-Tsun Wu Key Laboratory of CAS, 96 Jinzhai Road, Hefei, 230026, Anhui, PR China}}
\affil[d]{\footnotesize{School of Mathematical Sciences, Anhui University, 111 Jiulong Road, Hedei, 230039, Anhui, PR China.}}
\date{}
\maketitle
\newcommand\blfootnote[1]{%
\begingroup
\renewcommand\thefootnote{}\footnote{#1}%
\addtocounter{footnote}{-1}%
\endgroup}
\blfootnote{E-mail addresses: cmy1325@163.com (M.Y. Cao), koolen@ustc.edu.cn (J.H. Koolen), piez@naver.com (J.Y. Yang).}

\renewcommand\abstractname{Abstract}
\begin{abstract} In this note, we study the maximum number $N_\alpha(d)$ of a system of equiangular lines in $\mathbb{R}^d$ with cosine $\alpha$, where
$\frac{1}{\alpha}$ is not an odd positive integer. This note is motivated by a remark in a $2018$ paper by Balla, Dr\"{a}xler, Keevash and Sudakov.\\
\emph{key words}: Equiangular lines, Seidel matrix, smallest eigenvalue, spectral radius, rank. \end{abstract}

\section{Introduction}

A system of lines through the origin in $d$-dimensional Euclidean space $\mathbb{R}^d$ is called equiangular if the angle between any pair of lines is
the same. Let $N_{\alpha}(d)$ be the maximum number of a system of equiangular lines in $\mathbb{R}^d$ with common angle $\arccos \alpha$. For more information  on systems of equiangular lines, see \cite{GKMS}. In 2018, Balla, Dr\"{a}xler, Keevash and Sudakov \cite{BDKS} conjectured  that if
$\alpha = \frac{1}{2r-1}$ for a positive integer $r \geq 2$, then $N_{\frac{1}{2r-1}} (d) =\frac{r(d-1)}{r-1} + O(1)$, for sufficiently large $d$. They also wrote `If $\alpha$ is not of the above form, the situation is less clear but it is conceivable that $N_{\alpha}(d) = (1 + o(1))d$.'.

In this note, we will consider $N_{\alpha}(d)$ when $\frac{1}{\alpha}$ is not an odd integer. We prove the following:

\begin{theorem}\label{N}
Let $0<\tau\leq\frac{1}{1+2\sqrt{2+\sqrt{5}}}$. Then there exists a sequence $(\alpha_i)^{\infty}_{i=1}$ such that
\begin{enumerate}[(i)]
\item $\underset{i\rightarrow\infty}{\lim}{\alpha_i}=\tau$,

\item for all $i$, there exists $\eta_i>0$ such that $N_{\alpha_i}(d)\geq(1+\eta_i)d$.
\end{enumerate}
\end{theorem}

\section{Seidel matrices}
First, we introduce Seidel matrix which is our main tool. All graphs are simple and undirected. For undefined terminologies, we refer to \cite{GR, BH}.

Let $G$ be a graph with $n$ vertices. The adjacency matrix $A(G)$ of $G$ is an $n \times n$ matrix whose rows and columns are indexed by the vertices of $G$ such that $A(G)_{xy}$ is $1$ if $x$ and $y$ are adjacent vertices and $0$ otherwise. The Seidel matrix $S(G)$ of $G$ is the matrix $S(G) = \mathbf{J-I}-A(G)$, where $\mathbf{J}$ is the all-ones matrix and $\mathbf{I}$ is the identity matrix. Let $\mathbf{j}$ denote
the all-ones vector. The spectral radius $\rho(G)$ of a graph $G$ is the largest eigenvalue of $A(G)$.

Seidel matrices and systems of equiangular lines, are related as follow (see for example, \cite[Section 11.1]{GR}):

\begin{proposition}\label{eq}

There exists a system of $n$ equiangular lines in $\mathbb{R}^d$ with common angle $\arccos \alpha$ if and only if there exists a graph $G$ with $n$ vertices such that $S(G)$ has smallest eigenvalue at least $-\frac{1}{\alpha}$ and rk$(S(G) + \frac{1}{\alpha}\mathbf{I}) = d$.
\end{proposition}

This leads to the following definition. The number $R_{\beta}(n)$ is defined as
$R_\beta(n):=\min\{\rk(S(G)+\beta\mathbf{I})\mid G$ is a graph on $n$ vertices and $S(G)$ has smallest eigenvalue at least $-\beta\}$. We obtain the following lemma immediately from Proposition \ref{eq}.

\begin{lemma}\label{Ne}
$N_{\alpha}(d) \geq n$ if and only if $R_{\frac{1}{\alpha}}(n) \leq d.$
\end{lemma}

\begin{example}
 $N_{\frac{1}{3}}(d)=28$  for $7\leq d\leq13$ and $R_3(n)=7$ for $17\leq n\leq28$ (see \cite[Table $1$]{GKMS}).
\end{example}
Now we discuss some properties of $R_{\beta}(n)$. Since $S(G)$ is a real symmetric matrix, it is diagonalizable and all its eigenvalues are totally real algebraic integers. This implies that if
$\beta>1$ is not a totally real algebraic integer, then $R_\beta(n)=n$ and $N_{\frac{1}{\beta}}(d)=d$ for all integers $n\geq2$ and $d\geq2$.

Haemers observed that for a graph $G$ of order $n$, we have $\det(-S(G)+x\mathbf{I})\equiv\det(x\mathbf{I}-\mathbf{J}+\mathbf{I})\ (\rm{mod}\ 2)$ (see \cite{G}). This shows that for example $R_{\sqrt{2}}(n)=n$ and $N_{\frac{1}{\sqrt{2}}}(d)=d$ for all integers $d\geq2$ and $n\geq2$. Similar but more complicated formulas for the coefficients of
$\det(-S(G)+x\mathbf{I})$ are known (see for example, \cite{G,GKMS,GY}). The observation of Haemers also implies that $2t+1\geq R_{2m}(2t+1)\geq2t=R_{2m}(2t)$ and $N_{\frac{1}{2m}}(2t+1)=2t+1\geq N_{\frac{1}{2m}}(2t)\geq2t$, where $t>m$ are positive integers. Now, we give an example of the case $R_{2m}(2t+1)=2t$ and $N_{\frac{1}{2m}}(2t)=2t+1$ for certain $m$ and $t$. Let $G$ be a $3$-regular graph on $4n'+2$ vertices, where $n'\geq2$ is an
integer. Then the Seidel matrix of the complement of the line graph of $G$ has smallest eigenvalue $6-6n'$ with multiplicity one. This shows
$R_{6n'-6}(6n'+3)=6n'+2$ and $N_{\frac{1}{6n'-6}}(6n'+2)=6n'+3$.

Next, we show the following lemma on Seidel matrices.
\begin{lemma}\label{union}
Let $G$ be a graph of order $n$ with $t$ connected components $H_1,H_2,\ldots,H_t$. If $\rho(H_i)=\rho$ for all $i$, then $S(G)$ has smallest
eigenvalue $-2\rho-1$ with multiplicity at least $t-1$.
\end{lemma}
\begin{proof}
Let $B(G)=S(G)-\mathbf{J}$. We can check that $B(G)$ is the block diagonal matrix of the form
\begin{gather*}
\begin{bmatrix}
-2A(H_1)-\mathbf{I}\\
& -2A(H_2)-\mathbf{I} & & \\
&  &  \ddots &\\
& & & -2A(H_t)-\mathbf{I}
\end{bmatrix}.
\end{gather*}
The matrix $B(G)$ has smallest eigenvalue $-2\rho-1$ with multiplicity $t$, so there are $t-1$ linearly independent eigenvectors
$\mathbf{v}_1,\ldots,\mathbf{v}_{t-1}$ of $B(G)$ such that $B(G)\mathbf{v}_i=-2\rho-1$ and $\langle\mathbf{v}_i,\mathbf{j}\rangle=0$, for
$i=1,\ldots,t-1$. It follows that $S(G)$ has eigenvalue $-2\rho-1$ with multiplicity at least $t-1$. This shows the lemma.
\end{proof}

To show Theorem \ref{N}, we need the following result of Shearer \cite{Shearer} on spectral radius of graphs at least $\sqrt{2+\sqrt5}$.

\begin{theorem}\label{Shearer}
For any real number $\lambda\geq\sqrt{2+\sqrt{5}}$, there exists a sequence of graphs $(G_i)^{\infty}_{i=1}$ such that
$\underset{i\rightarrow\infty}{\lim}\rho(G_i)=\lambda$.
\end{theorem}

For the classification of the graphs with spectral radius less than $\sqrt{2+\sqrt5}$, see \cite{BN}.\\

Now, we will prove the following theorem.

\begin{theorem}\label{R}
For a real number $\mu\geq1+2\sqrt{2+\sqrt{5}}$, there exists a sequence $(\beta_i)^{\infty}_{i=1}$ such that the following hold:
\begin{enumerate}[(i)]
\item $\underset{i\rightarrow\infty}{\lim}{\beta_i}=\mu$,
\item for all $i$, there exists $\varepsilon_i>0$ such that $R_{\beta_i}(n)\leq(1-\varepsilon_i)n$.
\end{enumerate}
\end{theorem}

\begin{proof}
Let $\mu$ be a real number at least $1+2\sqrt{2+\sqrt{5}}$. For $\lambda:=\frac{\mu-1}{2} \geq \sqrt{2+\sqrt5}$, there exists a sequence of graphs $(G_i)^{\infty}_{i=1}$ such that $\underset{i\rightarrow\infty}{\lim}\rho(G_i)=\lambda$ by Theorem \ref{Shearer}.
Let $tG_i$ be the disjoint union of $t$ copies of $G_i$ for a positive integer $t$. For real numbers $\beta_i=2\rho(G_i)+1$ and $n_i=n(G_i)$, $tG_i$ is a graph with $tn_i$ vertices and the smallest eigenvalue of $S(tG_i)$ is $-\beta_i$ with multiplicity at least $t-1$ by Lemma \ref{union}. It follows that $R_{\beta_i}(tn_i)\leq tn_i-t+1$ by Proposition \ref{eq}. Since $\underset{i\rightarrow\infty}{\lim}{\beta_i}=\mu$, this finishes the proof.
\end{proof}

Theorem \ref{N} follows now from Theorem \ref{R} and Lemma \ref{Ne}.

\section*{Acknowledgments}

M.Y. Cao is partially supported by the National Natural Science Foundation of China (No. $11571044$ and No. $61373021$) and the Fundamental Research Funds
for the Central Universities.

J.H. Koolen is partially supported by the National Natural Science Foundation of China (No. $11471009$ and No. $11671376$) and Anhui Initiative of Quantum
Information Technologies (No. AHY $150000$).

J.Y. Yang is partially supported by the National Natural Science Foundation of China (No. $11371028$).

\end{document}